\documentclass[a4paper,12pt,final]{amsart}
\usepackage{times,a4wide,mathrsfs,amssymb,amsmath,amsthm}

\newcommand{\C}{\mathbb{C}}
\newcommand{\ZZ}{\mathbb{Z}}

\newcommand{\QQ}{\mathbb{Q}}
\newcommand{\NN}{\mathbb{N}}
\newcommand{\PP}{\mathbb{P}}

\newcommand{\OO}{\mathcal O}

\newcommand{\XX}{\mathcal X}
\newcommand{\YY}{\mathcal Y}

\newcommand{\VV}{\mathcal V}
\newcommand{\WW}{\mathcal W}

\newcommand{\EE}{\mathcal E}
\newcommand{\MM}{\mathcal M}

\newcommand{\FF}{\mathcal F}
\newcommand{\BB}{\mathfrak B}

\newcommand{\codim}{\hbox{codim}}

\newcommand{\gr}{\hbox{Gr}}
\newcommand{\wt}{\widetilde}
\newcommand{\rom}{\romannumeral}

\DeclareMathOperator{\aut}{Aut}
\DeclareMathOperator{\ide}{id}

\DeclareMathOperator{\ima}{Im}

\newtheorem{theorem}{Theorem}[section]

\newtheorem{lemma}[theorem]{Lemma}

\newtheorem{corollary}[theorem]{Corollary}
\newtheorem{proposition}[theorem]{Proposition}
\newtheorem{conjecture}[theorem]{Conjecture}
\newtheorem{remark}[theorem]{Remark}
\newtheorem{definition}[theorem]{Definition}
\newtheorem{convention}{Conventions}

\newtheorem{nonumbering}{Theorem}

\newtheorem{nonumberingc}{Corollary}

\newtheorem{nonumberingt}{Acknowledgements}

\begin{document}
\author[Robert Laterveer]
{Robert Laterveer}

\address{Institut de Recherche Math\'ematique Avanc\'ee,
CNRS -- Universit\'e 
de Strasbourg,\
7 Rue Ren\'e Des\-car\-tes, 67084 Strasbourg CEDEX,
FRANCE.}
\email{robert.laterveer@math.unistra.fr}

\title[On Chow groups of HK fourfolds with non--symplectic involution II]{On the Chow groups of some hyperk\"ahler fourfolds with a non--symplectic involution II}

\begin{abstract} This article is about hyperk\"ahler fourfolds $X$ admitting a non--symplectic involution $\iota$. The Bloch--Beilinson conjectures predict the way $\iota$ should act on certain pieces of the Chow groups of $X$.
The main result is a verification of this prediction for Fano varieties of lines on certain cubic fourfolds. This has some interesting consequences for the Chow ring of the quotient $X/\iota$.
\end{abstract}


\maketitle

\section{Introduction}

For a smooth projective variety $X$ defined over $\C$, let $A^i(X):=CH^i(X)_{\QQ}$ denote the Chow groups of $X$ (i.e. the groups of codimension $i$ algebraic cycles on $X$ with $\QQ$--coefficients, modulo rational equivalence). As explained for instance in \cite{J2} or \cite{Vo} or \cite{MNP}, the Bloch--Beilinson conjectures form a beautiful crystal ball, allowing to make strikingly concrete predictions about Chow groups. 
%(one of these predictions being Bloch's famous conjecture \cite{B}). 
In this article, we consider one instance of such a prediction, concerning non--symplectic involutions on hyperk\"ahler varieties.

Let $X$ be a hyperk\"ahler variety (that is, a projective irreducible holomorphic symplectic manifold, cf. \cite{Beau0}, \cite{Beau1}), and assume that $X$ admits an anti--symplectic involution $\iota$. 
The action of $\iota$ on the subring $H^{\ast,0}(X)$ is well--understood: one has
 \[ \begin{split} \iota^\ast=  -\ide \colon\ \ \ &H^{2i,0}(X)\ \to\ H^{2i,0}(X)\ \ \ \hbox{for}\ i\ \hbox{odd}\ ,\\
                        \iota^\ast=  \ide \colon\ \ \ &H^{2i,0}(X)\ \to\ H^{2i,0}(X)\ \ \ \hbox{for}\ i\ \hbox{even}\ .\\
                  \end{split}\]
                  
The action of $\iota$ on the Chow ring $A^\ast(X)$ is more mysterious.                      
To state the conjectural behaviour,
we will now assume the Chow ring of $X$ has a bigraded ring structure $A^\ast_{(\ast)}(X)$, where each $A^i(X)$ splits into pieces
  \[ A^i(X) =\bigoplus_j A^i_{(j)}(X)\ ,\]
and the piece $A^i_{(j)}(X)$ is isomorphic to the graded $\gr^j_F A^i(X)$ for the Bloch--Beilinson filtration that conjecturally exists for all smooth projective varieties.   
 (It is expected that such a bigrading $A^\ast_{(\ast)}(-)$ exists for all hyperk\"ahler varieties \cite{Beau3}.) 
 
 Since the pieces $A^i_{(i)}(X)$ and $A^{\dim X}_{(i)}(X)$ should only depend on the subring $H^{\ast,0}(X)$, we arrive at the following conjecture:
 
 \begin{conjecture}\label{conj} Let $X$ be a hyperk\"ahler variety of dimension $2m$, and let $\iota\in\aut(X)$ be a non--symplectic involution. Then
   \[  \begin{split}  \iota^\ast= (-1)^i \ide\colon\ \ \ &A^{2i}_{(2i)}(X)\ \to\ A^{2i}(X)\ ,\\
                           \iota^\ast= (-1)^i \ide\colon\ \ \ &A^{2m}_{(2i)}(X)\ \to\ A^{2m}(X)\ .\\
                    \end{split}\]       
  \end{conjecture}
 
 This conjecture is studied (and proven in some favourable cases) in \cite{EPW}, \cite{HKnonsymp}, \cite{ChowEPW}, \cite{I}, \cite{cube}.
 The aim of this article is to provide more examples where conjecture \ref{conj} is verified, by considering Fano varieties of lines on cubic fourfolds. The main result is as follows:

\begin{nonumbering}[=theorem \ref{main}] Let $Y\subset\PP^5(\C)$ be a smooth cubic fourfold defined by an equation
    \[ \begin{split} (X_0)^2 \ell_0(X_3,X_4,X_5)+  (X_1)^2\ell_1(X_3,X_4,X_5)+(X_2)^2\ell_2(X_3,X_4,X_5)+X_0 X_1  \ell_3(X_3,X_4,X_5)& \\ + X_0 X_2 \ell_4(X_3,X_4,X_5) +  
    X_1 X_2 \ell_5(X_3,X_4,X_5)  +g(X_3,\ldots,X_5)=0\ ,& \\
    \end{split}\]
    where the $\ell_i$ are linear forms and $g$ is a homogeneous degree $3$ polynomial.
 Let $X=F(Y)$ be the Fano variety of lines in $Y$. Let $\iota\in\aut(X)$ be the anti--symplectic involution induced by
  \[ \begin{split}  \PP^5(\C)\ &\to\ \PP^5(\C)\ ,\\
                        [X_0,X_1,\ldots,X_5]\ &\mapsto\ [-X_0,-X_1,-X_2,X_3,X_4,X_5]\ .\\
                        \end{split}\]
 Then
   \[  \begin{split}  \iota^\ast=-\ide\colon\ \ \ &A^i_{(2)}(X)\ \to\ A^i_{(2)}(X)\ \ \ \hbox{for}\ i=2,4\ ;\\
                            \iota^\ast=\ide\colon\ \ \ &A^4_{(j)}(X)\ \to\ A^4_{(j)}(X)\ \ \ \hbox{for}\ j=0,4\ .\\
                        \end{split}\]  
        \end{nonumbering}
        
 Here, the notation $A^\ast_{(\ast)}(X)$ refers to the Fourier decomposition of the Chow ring of $X$ constructed by Shen--Vial \cite{SV}. (We mention in passing that for $X$ as in theorem \ref{main}, it is unfortunately not yet known whether $A^\ast_{(\ast)}(X)$ is a bigraded ring, cf. remark \ref{pity} below.) 
 
 To prove theorem \ref{main}, we exploit the fact that the family of cubics under consideration is sufficiently large for the method of ``spread'' as developed by Voisin \cite{V0}, \cite{V1} to apply. There is only one other family of cubic fourfolds with a polarized involution that is anti-symplectic on the Fano variety; this other family has been treated in \cite{I}, using arguments very similar to those of the present article.
 The action of polarized {\em symplectic\/} automorphisms on Chow groups of Fano varieties of cubic fourfolds has already been studied by L. Fu \cite{LFu2}, similarly using Voisin's method of ``spread''.  
        
  Theorem \ref{main} has some rather striking consequences for the Chow ring of the quotient (this quotient is a slightly singular Calabi--Yau fourfold):
  
 \begin{nonumberingc}[=corollaries \ref{cor} and \ref{cor3}] Let $(X,\iota)$ be as in theorem \ref{main}, and let $Z:=X/\iota$ be the quotient. Then
  the images of the intersection product maps
   \[ \begin{split} A^2(Z)\otimes A^2(Z)\ &\to\ A^4(Z)\ ,\\
                         A^3(Z)\otimes A^1(Z)\ &\to\ A^4(Z)\ \\  
                        \end{split}   \]
   are of dimension $1$. 
       \end{nonumberingc}     
                          
In particular, this means that for any two cycles $b,c\in A^2(Z)$ (or $b\in A^3(Z)$ and $c\in A^1(Z)$), the $0$--cycle $b\cdot c$ is rationally trivial if and only if it has degree $0$. This is similar to results for Calabi--Yau complete intersections obtained in \cite{V13}, \cite{LFu}.

\begin{nonumberingc}[=corollary \ref{cor2}] Let $(X,\iota)$ be as in theorem \ref{main}, and let $Z:=X/\iota$ be the quotient.
 Then the image of the intersection product map
    \[ \ima \bigl(  A^2(Z)\otimes A^1(Z)\xrightarrow{} A^3(Z) \bigr) \]
  injects into $H^6(Z)$ under the cycle class map.
\end{nonumberingc}

%One expects a stronger statement to be true: conjecturally, for any $b$ in the image of $A^2(Z)\otimes A^1(Z)\to A^3(Z)$, one has that $b$ is rationally trivial if and only if $b$ is homologically trivial. Our argument does not allow to prove this, because of the nuisance (mentioned above) that it is not yet known whether $A^\ast_{(\ast)}(X)$ is a bigraded ring (cf. remark \ref{:(}).

Corollaries \ref{cor} and \ref{cor2} provide some support in favour of the conjecture that 
  \[ A^2_{hom}(Z)\stackrel{??}{=}0\ .\] 
  (The Bloch--Beilinson conjectures would imply that $A^2_{AJ}(M)=0$ for any Calabi--Yau variety $M$ of dimension $>2$. As far as I know, there is not a single Calabi--Yau variety $M$ for which this is known to be true.)

 \vskip0.6cm

\begin{convention} In this article, the word {\sl variety\/} will refer to a reduced irreducible scheme of finite type over $\C$. A {\sl subvariety\/} is a (possibly reducible) reduced subscheme which is equidimensional. 

{\bf All Chow groups will be with rational coefficients}: we will denote by $A_j(X)$ the Chow group of $j$--dimensional cycles on $X$ with $\QQ$--coefficients; for $X$ smooth of dimension $n$ the notations $A_j(X)$ and $A^{n-j}(X)$ are used interchangeably. 

The notations $A^j_{hom}(X)$, $A^j_{AJ}(X)$ will be used to indicate the subgroups of homologically trivial, resp. Abel--Jacobi trivial cycles.
For a morphism $f\colon X\to Y$, we will write $\Gamma_f\in A_\ast(X\times Y)$ for the graph of $f$.
The category of Chow motives (i.e., pure motives with respect to rational equivalence as in \cite{Sc}, \cite{MNP}) will be denoted $\MM_{\rm rat}$.

%To avoid heavy notation, if $\tau\colon Y\to X$ is a closed inclusion and $a\in A_iY$, we will frequently write $a\in A_iX$ to indicate the proper push--forward $\tau_\ast(a)$. Likewise, for any inclusion $Y\subset X$ and $b\in A^jX$ we will often write
%  \[   b\vert_{Y}\ \ \in A^jY\]
 % to indicate the cycle class $\tau^\ast(b)$.

%The Griffiths group $\grif^j$ is the group of codimension $j$ cycles that are homologically trivial modulo algebraic equivalence, again with $\QQ$--coefficients. 

We will write $H^j(X)$ 
to indicate singular cohomology $H^j(X,\QQ)$.

%Given a group $G\subset\aut(X)$ of automorphisms of $X$, we will write $A^j(X)^G$ (and $H^j(X)^G$) for the subgroup of $A^j(X)$ (resp. $H^j(X)$) invariant under $G$.
\end{convention}

\section{Preliminaries}

\subsection{MCK decomposition}
\label{ss1}

\begin{definition}[Murre \cite{Mur}] Let $X$ be a smooth projective variety of dimension $n$. We say that $X$ has a {\em CK decomposition\/} if there exists a decomposition of the diagonal
   \[ \Delta_X= \pi_0+ \pi_1+\cdots +\pi_{2n}\ \ \ \hbox{in}\ A^n(X\times X)\ ,\]
  such that the $\pi_i$ are mutually orthogonal idempotents and $(\pi_i)_\ast H^\ast(X)= H^i(X)$.
  
  (NB: ``CK decomposition'' is shorthand for ``Chow--K\"unneth decomposition''.)
\end{definition}

\begin{remark} The existence of a CK decomposition for any smooth projective variety is part of Murre's conjectures \cite{Mur}, \cite{J2}. 
%If a quotient variety $X$
%has finite--dimensional motive, and the K\"unneth components are algebraic, then $X$ has a CK decomposition (this can be proven just as \cite{J2}, where this is stated for smooth $X$).
\end{remark}

\begin{definition}[Shen--Vial \cite{SV}] Let $X$ be a smooth projective variety of dimension $n$. Let $\Delta_X^{sm}\in A^{2n}(X\times X\times X)$ be the class of the small diagonal
  \[ \Delta_X^{sm}:=\bigl\{ (x,x,x)\ \vert\ x\in X\bigr\}\ \subset\ X\times X\times X\ .\]
  An {\em MCK decomposition\/} is a CK decomposition $\{\pi^X_i\}$ of $X$ that is {\em multiplicative\/}, i.e. it satisfies
  \[ \pi^X_k\circ \Delta_X^{sm}\circ (\pi^X_i\times \pi^X_j)=0\ \ \ \hbox{in}\ A^{2n}(X\times X\times X)\ \ \ \hbox{for\ all\ }i+j\not=k\ .\]
  
 (NB: ``MCK decomposition'' is shorthand for ``multiplicative Chow--K\"unneth decomposition''.) 
  
 A {\em weak MCK decomposition\/} is a CK decomposition $\{\pi^X_i\}$ of $X$ that satisfies
    \[ \Bigl(\pi^X_k\circ \Delta_X^{sm}\circ (\pi^X_i\times \pi^X_j)\Bigr){}_\ast (a\times b)=0 \ \ \ \hbox{for\ all\ } a,b\in\ A^\ast(X)\ .\]
  \end{definition}
  
  \begin{remark} The small diagonal (seen as a correspondence from $X\times X$ to $X$) induces the {\em multiplication morphism\/}
    \[ \Delta_X^{sm}\colon\ \  h(X)\otimes h(X)\ \to\ h(X)\ \ \ \hbox{in}\ \MM_{\rm rat}\ .\]
 Suppose $X$ has a CK decomposition
  \[ h(X)=\bigoplus_{i=0}^{2n} h^i(X)\ \ \ \hbox{in}\ \MM_{\rm rat}\ .\]
  By definition, this decomposition is multiplicative if for any $i,j$ the composition
  \[ h^i(X)\otimes h^j(X)\ \to\ h(X)\otimes h(X)\ \xrightarrow{\Delta_X^{sm}}\ h(X)\ \ \ \hbox{in}\ \MM_{\rm rat}\]
  factors through $h^{i+j}(X)$.
  
  If $X$ has a weak MCK decomposition, then setting
    \[ A^i_{(j)}(X):= (\pi^X_{2i-j})_\ast A^i(X) \ ,\]
    one obtains a bigraded ring structure on the Chow ring: that is, the intersection product sends $A^i_{(j)}(X)\otimes A^{i^\prime}_{(j^\prime)}(X) $ to  $A^{i+i^\prime}_{(j+j^\prime)}(X)$.
    
      It is expected (but not proven !) that for any $X$ with a weak MCK decomposition, one has
    \[ A^i_{(j)}(X)\stackrel{??}{=}0\ \ \ \hbox{for}\ j<0\ ,\ \ \ A^i_{(0)}(X)\cap A^i_{hom}(X)\stackrel{??}{=}0\ ;\]
    this is related to Murre's conjectures B and D, which have been formulated for any CK decomposition \cite{Mur}.
        
  The property of having an MCK decomposition is severely restrictive, and is closely related to Beauville's ``(weak) splitting property'' \cite{Beau3}. For more ample discussion, and examples of varieties with an MCK decomposition, we refer to \cite[Section 8]{SV}, as well as \cite{V6}, \cite{SV2}, \cite{FTV}.
    \end{remark}

In what follows, we will make use of the following: 

\begin{theorem}[Shen--Vial \cite{SV}]\label{fanomck} Let $Y\subset\PP^5(\C)$ be a smooth cubic fourfold, and let $X:=F(Y)$ be the Fano variety of lines in $Y$. There exists a CK decomposition $\{\pi^X_i\}$ for $X$, and 
  \[ (\pi^X_{2i-j})_\ast A^i(X) = A^i_{(j)}(X)\ ,\]
  where the right--hand side denotes the splitting of the Chow groups defined in terms of the Fourier transform as in \cite[Theorem 2]{SV}. Moreover, we have
  \[ A^i_{(j)}(X)=0\ \ \ \hbox{for\ }j<0\ \hbox{and\ for\ }j>i\ .\]
  
  In case $Y$ is very general, the Fourier decomposition $A^\ast_{(\ast)}(X)$ forms a bigraded ring, and hence
  $\{\pi^X_i\}$ is a weak MCK decomposition.
    \end{theorem}

\begin{proof} (A remark on notation: what we denote $A^i_{(j)}(X)$ is denoted $CH^i(X)_j$ in \cite{SV}.)

The existence of a CK decomposition $\{\pi^X_i\}$ is \cite[Theorem 3.3]{SV}, combined with the results in \cite[Section 3]{SV} to ensure that the hypotheses of \cite[Theorem 3.3]{SV} are satisfied. (Alternatively, the existence of a CK decomposition is also established in \cite[Proposition A.6]{FLV}.) According to \cite[Theorem 3.3]{SV}, the given CK decomposition agrees with the Fourier decomposition of the Chow groups. The ``moreover'' part is because the $\{\pi^X_i\}$ are shown to satisfy Murre's conjecture B \cite[Theorem 3.3]{SV}.

The statement for very general cubics is \cite[Theorem 3]{SV}.
    \end{proof}

\begin{remark}\label{pity} Unfortunately, it is not yet known that the Fourier decomposition of \cite{SV} induces a bigraded ring structure on the Chow ring for {\em all\/} Fano varieties of smooth cubic fourfolds. For one thing, it has not yet been proven that 
  \[  A^2_{(0)}(X)\cdot A^2_{(0)}(X)\ \stackrel{??}{\subset}\  A^4_{(0)}(X) \] 
  for the Fano variety of a given (not necessarily very general) cubic fourfold
  (cf. \cite[Section 22.3]{SV} for discussion). 
  
To prove that $A^\ast_{(\ast)}()$ is a bigraded ring for all Fano varieties of smooth cubic fourfolds, it would suffice to construct an MCK decomposition for the Fano variety of the very general cubic fourfold.
\end{remark}

%\subsection{Refined CK decomposition}

%\begin{theorem}[Vial \cite{V2}] Let $X$ be a smooth projective variety of dimension $n$. Assume that $X$ has a CK decomposition $\{\pi^X_i\}$ and satisfies the standard Lefschetz conjecture $B(X)$. Then there exists a further splitting
%  \[ \pi^X_2 = \pi^X_{2,0} + \pi^X_{2,1}\ \ \ \hbox{in}\ A^n(X\times X)\ ,\]
%  where $\pi^X_{2,0}$ and $\pi^X_{2,1}$ are idempotents. The action on cohomology verifies
 % \[ (\pi^X_{2,0})_\ast H^\ast(X) = H^2_{tr}(X)\ ,\]
 % where $H^2_{tr}(X)\subset H^2(X)$ is defined as the smallest Hodge sub--structure containing $H^{2,0}(X)$. The action on Chow groups verifies
 % \[  (\pi^X_{2,0})_\ast A^2(X) = A^2_{AJ}(X)\ .\]
%  \end{theorem}

\subsection{A multiplicative result}
\label{ss2}

Let $X$ be the Fano variety of lines on a smooth cubic fourfold. As we have seen (theorem \ref{fanomck}), the Chow ring of $X$ splits into pieces $A^i_{(j)}(X)$.
The work \cite{SV} contains a thorough analysis of the multiplicative behaviour of these pieces. Here are the relevant results we will be needing:

\begin{theorem}[Shen--Vial \cite{SV}]\label{chowringfano} Let $Y\subset\PP^5(\C)$ be a smooth cubic fourfold, and let $X:=F(Y)$ be the Fano variety of lines in $Y$. 

\noindent
(\rom1) There exists $\ell\in A^2_{(0)}(X)$ such that intersecting with $\ell$ induces an isomorphism
  \[ \cdot\ell\colon\ \ \ A^2_{(2)}(X)\ \xrightarrow{\cong}\ A^4_{(2)}(X)\ .\]

\noindent
(\rom2) Intersection product induces a surjection
  \[ A^2_{(2)}(X)\otimes A^2_{(2)}(X)\ \twoheadrightarrow\ A^4_{(4)}(X)\ .\]
\end{theorem} 
     
 \begin{proof} Statement (\rom1) is \cite[Theorem 4]{SV}. Statement (\rom2) is \cite[Proposition 20.3]{SV}.
  \end{proof}

  \subsection{The involution}
  
  \begin{lemma}\label{inv} Let $\iota_\PP\in\aut(\PP^5(\C))$ be the involution defined as
    \[  [X_0,X_1,\ldots,X_5]\ \mapsto\ [-X_0,-X_1,-X_2,X_3,X_4,X_5]\ .\]
  The cubic fourfolds invariant under $\iota_\PP$ are exactly those defined by an equation  
   \[ \begin{split} (X_0)^2 \ell_0(X_3,X_4,X_5)+  (X_1)^2\ell_1(X_3,X_4,X_5)+(X_2)^2\ell_2(X_3,X_4,X_5)+X_0 X_1  \ell_3(X_3,X_4,X_5)& \\ + X_0 X_2 \ell_4(X_3,X_4,X_5) +  
    X_1 X_2 \ell_5(X_3,X_4,X_5)  +g(X_3,\ldots,X_5)=0\ ,& \\
    \end{split}\]
    where the $\ell_i$ are linear forms and $g$ is a homogeneous degree $3$ polynomial. 
    
   Let $Y\subset\PP^5(\C)$ be a smooth cubic invariant under $\iota_\PP$, and let $\iota_Y\in\aut(Y)$ be the involution induced by $\iota_\PP$. Let $X=F(Y)$ be the Fano variety of lines in $Y$, and let $\iota\in\aut(X)$ be the involution induced by $\iota_Y$. The involution $\iota$ is anti--symplectic.
      \end{lemma}   
      
   \begin{proof} The only thing that needs explaining is the last phrase; this is proven in \cite[Section 7]{Cam}. The idea is that there is an isomorphism of Hodge structures, compatible with the involution
   \[ H^2(X)\cong H^4(Y)\ .\]
   The action of $\iota_Y$ on $H^{3,1}(Y)$ is minus the identity, because $H^{3,1}(Y)$ is generated by the meromorphic form
   \[  \sum_{i=0}^5  (-1)^i X_i  {dX_0\wedge \ldots\wedge d\hat{X_i}\wedge\ldots\wedge dX_5 \over  f^2}\ ,\]
   where $f$ is an equation for $Y$.
   \end{proof}

 \subsection{Spread}
 \label{ssvois}
 
    \begin{lemma}[Voisin \cite{V0}, \cite{V1}]\label{projbundle} Let $M$ be a smooth projective variety of dimension $n+1$, and $L$ a very ample line bundle on $M$. Let 
    \[ \pi\colon \XX\to B\]
    denote a family of hypersurfaces, where $B\subset\vert L\vert$ is a Zariski open.
      Let
   \[   p\colon \wt{\XX\times_B \XX}\ \to\ \XX\times_B \XX\]
   denote the blow--up of the relative diagonal. 
 Then $\wt{\XX\times_B \XX}$ is Zariski open in $V$, where $V$ is a projective bundle over $\wt{M\times M}$, the blow--up of $M\times M$ along the diagonal.
   \end{lemma} 
  
  \begin{proof} This is \cite[Proof of Proposition 3.13]{V0} or \cite[Lemma 1.3]{V1}. The idea is to define $V$ as
   \[  V:=\Bigl\{ \bigl((x,y,z),\sigma\bigr) \ \vert\ \sigma\vert_z=0\Bigr\}\ \ \subset\ \wt{M\times M}\times \vert L\vert\ .\]
   The very ampleness assumption ensures $V\to\wt{M\times M}$ is a projective bundle.
    \end{proof}

  This is used in the following key proposition: 
   
     \begin{proposition}[Voisin \cite{V1}]\label{voisin1} Assumptions as in lemma \ref{projbundle}. Assume moreover $M$ has trivial Chow groups. Let $R\in A^n(V)_{}$. Suppose that for all $b\in B$ one has
    \[ H^n(X_b)_{prim}\not=0\ \ \ \ 
  \hbox{and}\ \ \ \ 
     R\vert_{\wt{X_b\times X_b}}=0\ \ \in H^{2n}(\wt{X_b\times X_b})\ .\]
   Then there exists $\gamma\in A^n(M\times M)_{}$ such that
    \[     (p_b)_\ast \bigl(R\vert_{\wt{X_b\times X_b}}\bigr)= \gamma\vert_{X_b\times X_b}  \ \ \in A^{n}({X_b\times X_b})_{}\]  
    for all $b\in B$. 
   (Here $p_b$ denotes the restriction of $p$ to $\wt{X_b\times X_b}$, which is the blow--up of $X_b\times X_b$ along the diagonal.)
    \end{proposition}

\begin{proof} This is \cite[Proposition 1.6]{V1}.
\end{proof}

 The following is an equivariant version of proposition \ref{voisin1}:
 
  \begin{proposition}[Voisin \cite{V1}]\label{voisin2} Let $M$ and $L$ be as in proposition \ref{voisin1}. Let $G\subset\aut(M)$ be a finite group. Assume the following:
  
  \noindent
  (\rom1) The linear system $\vert L\vert^G:=\PP\bigl( H^0(M,L)^G\bigr)$ has no base--points, and the locus of points in $\wt{M\times M}$ parametrizing triples $(x,y,z)$ such that the length $2$ subscheme $z$ imposes only one condition on $\vert L\vert^G$ is contained in the union of (proper transforms of) graphs of non--trivial elements of $G$, plus some loci of codimension $>n+1$.
  
  \noindent
  (\rom2) Let $B\subset\vert L\vert^G$ be the open parametrizing smooth hypersurfaces, and let $X_b\subset M$ be a hypersurface for $b\in B$ general. There is no non--trivial relation
   \[ {\displaystyle\sum_{g\in G}} c_g \Gamma_g +\gamma=0\ \ \ \hbox{in}\ H^{2n}(X_b\times X_b)\ ,\]
   where $c_g\in\QQ$ and $\gamma$ is a cycle in $\ima\bigl( A^n(M\times M)\to A^n(X_b\times X_b)\bigr)$.
   
   Let $R\in A^n(\XX\times_B \XX)$ be such that
     \[  R\vert_{{X_b\times X_b}}=0\ \ \in H^{2n}({X_b\times X_b})\ \ \ \forall b\in B\ .\]
    Then there exists $\gamma\in A^n(M\times M)_{}$ such that
    \[     R\vert_{{X_b\times X_b}}= \gamma\vert_{X_b\times X_b}  \ \ \in A^{n}({X_b\times X_b})\ \ \ \forall b\in B\ .\]  
  \end{proposition} 
  
 \begin{proof} This is not stated verbatim in \cite{V1}, but it is contained in the proof of \cite[Proposition 3.1 and Theorem 3.3]{V1}. Let us very briefly review the argument.
 One considers the variety
   \[  V:=\Bigl\{ \bigl((x,y,z),\sigma\bigr) \ \vert\ \sigma\vert_z=0\Bigr\}\ \ \subset\ \wt{M\times M}\times \vert L\vert^G\ .\]   
   The problem is that $V$ is no longer a projective bundle over $\wt{M\times M}$. However, as explained in the proof of \cite[Theorem 3.3]{V1}, hypothesis (\rom1) ensures that one 
   can obtain a projective bundle after blowing up the graphs $\Gamma_g, g\in G$ plus some loci of codimension $>n+1$. Let $M^\prime\to\wt{M\times M}$ denote the result of these blow--ups, and let $V^\prime\to M^\prime$ denote the projective bundle obtained by base--change. 
% For any $b\in B$, let $\tau_b\colon \wt{X_b\times X_b}\to X_b\times X_b$ denote the restriction of $\tau\colon M^\prime\to M\times M$ to $\tau^{-1}(X_b\times X_b)$. 

Analyzing the situation as in \cite[Proof of Theorem 3.3]{V1}, one obtains
   \[ R\vert_{X_b\times X_b} =R_0\vert_{X_b\times X_b}+ {\displaystyle\sum_{g\in G}} \lambda_g \Gamma_g\ \ \ \hbox{in}\ A^n(X_b\times X_b) \ ,\]
   where $R_0\in A^n(M\times M)$ and $\lambda_g\in\QQ$ (this is \cite[Equation (15)]{V1}).
   By assumption, $R\vert_{X_b\times X_b}$ is homologically trivial. In view of hypothesis (\rom2), this implies that all $\lambda_g$ must be $0$.   
     \end{proof}

\section{Main result}

This section contains the proof of the main result of this note, theorem \ref{main}. The proof is split in two parts. In the first part, we prove a statement (theorem \ref{main2}) about the action of the involution on $1$--cycles on the cubic $Y$. The proof is an application of the technique of ``spread'' of cycles in a family, as developed by Voisin \cite{V0}, \cite{V1}, \cite{V8}, \cite{Vo} (more precisely, the results recalled in subsection \ref{ssvois}).

In the second part, we deduce from this our main result, theorem \ref{main}. This second part builds on the structural results of Shen--Vial \cite{SV} (notably the results recalled in subsections \ref{ss1} and \ref{ss2}).
     
   \subsection{First part} 
   
   \begin{theorem}\label{main2}  Let $Y\subset\PP^5(\C)$ be a smooth cubic fourfold defined by an equation
    \[ \begin{split} (X_0)^2 \ell_0(X_3,X_4,X_5)+  (X_1)^2\ell_1(X_3,X_4,X_5)+(X_2)^2\ell_2(X_3,X_4,X_5)+X_0 X_1  \ell_3(X_3,X_4,X_5)& \\ + X_0 X_2 \ell_4(X_3,X_4,X_5) +  
    X_1 X_2 \ell_5(X_3,X_4,X_5)  +g(X_3,\ldots,X_5)=0\ ,& \\
    \end{split}\]
    where the $\ell_i$ are linear forms and $g$ is a homogeneous degree $3$ polynomial.   
   
    Let $\iota_Y\in\aut(Y)$ be the involution of lemma \ref{inv}. Then
    \[  (\iota_Y)^\ast = -\ide\colon\ \ \ A^3_{hom}(Y)\ \to\ A^3(Y)\ .\]
    \end{theorem}

  \begin{proof} We have seen (proof of lemma \ref{inv}) that
    \[ (\iota_Y)^\ast =-\ide\colon\ \ \ H^{3,1}(Y)\ \to\ H^{3,1}(Y)\ .\]
  Let $H^4_{tr}(Y)$ denote the orthogonal complement (under the cup--product pairing) of $N^2 H^4(Y)$ (which coincides with $H^{2,2}(Y,\QQ)$ since the Hodge conjecture is true for $Y$). Since $H^4_{tr}(Y)\subset H^4(Y)$ is the smallest Hodge substructure containing $H^{3,1}(Y)$, we must also have
   \begin{equation}\label{van} (\iota_Y)^\ast =-\ide\colon\ \ \ H^{4}_{tr}(Y)\ \to\ H^{4}_{tr}(Y)\ .\end{equation}
   This implies that there is a decomposition
   \begin{equation}\label{fibrewise} {}^t \Gamma_{\iota_Y} = -\Delta_Y +\gamma\ \ \ \hbox{in}\ H^8(Y\times Y)\ ,\end{equation}
   where $\gamma\in A^4(Y\times Y)$ is a ``completely decomposed'' cycle, i.e.
    \[ \gamma=\gamma_0 + \gamma_2 +\gamma_4 +\gamma_6 +\gamma_8\ ,\]
    and $\gamma_{2i}$ has support on $V_i\times W_i\subset Y\times Y$ with $\dim V_i=i$ and $\dim W_i=4-i$.
    (Indeed, the cycle $\gamma$ is obtained by considering
      \[ \gamma_{2i}:= ({}^t \Gamma_{\iota_Y} +\Delta_Y)\circ \pi_{2i}\ \ \ \in H^8(Y\times Y)\ ,\]
      where $\pi_i$ denotes the K\"unneth component.
    For $i\not=4$, the claimed support condition is obviously satisfied since it is satisfied by $\pi_i$. For $i=4$, one uses (\ref{van}) to see that $\gamma_4$ is supported on $N^2 H^4(Y)\otimes N^2 H^4(Y)\subset H^8(Y\times Y)$.)
    
  We now consider things family--wise. Let 
    \[ \YY\ \to\ B \]
    denote the universal family of all smooth cubic fourfolds defined by an equation as in theorem \ref{main2}. Let $Y_b\subset\PP^5(\C)$
    denote the fibre over $b\in B$.
    
    The involution $\iota_\PP$ defines an involution $\iota_\YY\in\aut(\YY)$ by restriction. Let $\Delta_\YY\in A^4(\YY\times_B \YY)$ denote the relative diagonal. Obviously the argument leading to the decomposition (\ref{fibrewise}) applies to each fibre $Y_b$. This means that for each $b\in B$, there exists a completely decomposed cycle $\gamma_b\in A^4(Y_b\times Y_b)$ such that
    \[  \bigl({}^t \Gamma_{\iota_\YY} +\Delta_\YY\bigr)\vert_{Y_b\times Y_b} =\gamma_b\ \ \ \hbox{in}\ H^8(  Y_b\times Y_b)\ .\]
    Applying the ``spread'' result \cite[Proposition 3.7]{V0}, we can find a ``completely decomposed'' relative correspondence $\gamma\in A^4(\YY\times_B \YY)$
    such that
    \[  \bigl({}^t \Gamma_{\iota_\YY} +\Delta_\YY-\gamma\bigr)\vert_{Y_b\times Y_b} =0\ \ \ \hbox{in}\ H^8(  Y_b\times Y_b)\ \ \ \forall b\in B\ .\]
    (By this, we mean the following: there exist subvarieties $\VV_i, \WW_i\subset \YY$ for $i=0,2,4,6,8$ with 
      \[\codim \VV_i +\codim \WW_i=4\ ,\] 
      and such that the cycle $\gamma$ is
    supported on
      \[ \cup_i \VV_i\times_B \WW_i\ \ \ \subset \YY\times_B \YY\ .\]
      Actually, for $i\not=4$ this is obvious since the $\pi_i, i\not=4$ obviously exist relatively. The recourse to \cite[Proposition 3.7]{V0} can thus be limited to $i=4$.)
      
   That is, the relative correspondence
    \[ \Gamma:=    {}^t \Gamma_{\iota_\YY} +\Delta_\YY-\gamma \ \ \ \in A^4(\YY\times_B \YY) \]
    is fibrewise homologically trivial:
    \[ \Gamma\vert_{Y_b\times Y_b}=0\ \ \ \hbox{in}\ H^8(Y_b\times Y_b)\ \ \ \forall b\in B\ .\]
  
  At this point, we note that the family $\YY\to B$ is large enough to verify the hypotheses of proposition \ref{voisin2}; this will be proven in lemma \ref{ok} below. Applying proposition \ref{voisin2} to the relative correspondence $\Gamma$, we find that there exists $\delta\in A^4(\PP^5\times \PP^5)$ such that
   \[  \Gamma\vert_{Y_b\times Y_b} +\delta\vert_{Y_b\times Y_b} =0\ \ \ \hbox{in}\ A^4(Y_b\times Y_b)\ \ \ \forall b\in B\ .\]       
   But  
    \[   (\delta\vert_{Y_b\times Y_b})_\ast =0\colon\ \ \ A^3_{hom}(Y_b)\ \to\ A^3(Y_b)\ \ \ \forall b\in B \]
    (indeed, the action factors over $A^4_{hom}(\PP^5)$ which is $0$). Also, we have
    \[   (\gamma\vert_{Y_b\times Y_b})_\ast =0\colon\ \ \ A^3_{hom}(Y_b)\ \to\ A^3(Y_b)\ \ \ \hbox{for\ general\ } b\in B \]
    (indeed, for general $b\in B$ the restriction $\gamma\vert_{Y_b\times Y_b}$ is a completely decomposed cycle; such cycles do not act on $A^3_{hom}$ for dimension reasons).
    
  By definition of $\Gamma$, this means that
  \[ \bigl(  {}^t \Gamma_{\iota_{Y_b}} +\Delta_{Y_b}\bigr){}_\ast=0\colon\ \ \ A^3_{hom}(Y_b)\ \to\ A^3(Y_b)\ \ \ \hbox{for\ general\ } b\in B\ . \]
  This proves theorem \ref{main2} for general $b\in B$. To extend to all $b\in B$, one can reason as in \cite[Lemma 3.1]{LFu2}.
  
  It only remains to check that the hypotheses of Voisin's result are satisfied:
  
  \begin{lemma}\label{ok} Let $\YY\to B$ be the family of smooth cubic fourfolds as in theorem \ref{main2}, i.e.
    \[ B\ \subset\ \Bigl(\PP H^0\bigl(\PP^5,\OO_{\PP^5}(3)\bigr)\Bigr)^G \]
    is the open subset parametrizing smooth $G$--invariant cubics, where $G=\{id,\iota_\PP\}\subset\aut(\PP^5)$ is as above. This set--up verifies the hypotheses of proposition \ref{voisin2}.  
  \end{lemma}
  
 \begin{proof} 
 Let us first prove that hypothesis (\rom1) of proposition \ref{voisin2} is satisfied. 
  
  To this end, we consider the quotient morphisms
   \[ p\colon\ \ \PP^5\ \to\ P:=\PP^5/G\ \to\      P^\prime:=\PP(2^3,1^3)=\PP^5/(\ZZ/2\ZZ\times \ZZ/2\ZZ\times \ZZ/2\ZZ)\ ,\]
   where $P^\prime:=\PP(2^3,1^3)$ denotes a weighted projective space. 
   
   Let us write $\iota_0, \iota_1, \iota_2$ for the involutions of $\PP^5$ 
   \[  \begin{split}  &\iota_0 [X_0:\ldots:X_5] := [-X_0:X_1:\ldots:X_5]\ ,\\
                               &\iota_1 [X_0:\ldots:X_5] := [X_0:-X_1:X_2: \ldots :X_5]\ ,\\
                            &\iota_2 [X_0:\ldots:X_5] := [X_0:X_1:-X_2:X_3:X_4:X_5]\ .\\
                        \end{split} \] 
      (We note that $\iota_\PP=\iota_0\circ\iota_1\circ\iota_2$, and the weighted projective space $P^\prime$ is $\PP^5/ <\iota_0,\iota_1,\iota_2>$.)

   The sections in $\bigl(\PP H^0\bigl(\PP^5,\OO_{\PP^5}(3)\bigr)\bigr)^G$ are in bijection with sections coming from $P$, and contain the sections coming from $P^\prime$:  
    \[  \Bigl(\PP H^0\bigl(\PP^5,\OO_{\PP^5}(3)\bigr)\Bigr)^G \ \supset\  \PP H^0\bigl( P^\prime,\OO_{P^\prime}(3)\bigr)\ .\]  
   
   Let us now assume $x,y\in\PP^5$ are two points such that
   \[ (x,y)\not\in \Delta_{\PP^5}\cup  \bigcup_{0\le r_0,r_1,r_2\le 1}\Gamma_{(\iota_0)^{r_0}\circ (\iota_1)^{r_1}\circ (\iota_2)^{r_2}}\ .\]
   Then 
   \[ p(x)\not=p(y)\ \ \ \hbox{in}\ P^\prime\ ,\]
   and so (using lemma \ref{delorme} below) there exists $\sigma\in\PP H^0\bigl(P^\prime,\OO_{P^\prime}(3)\bigr)$ containing $p(x)$ but not $p(y)$. The pullback $p^\ast(\sigma)$ contains $x$ but not $y$, and so these points $(x,y)$ impose $2$ independent conditions on  $ \bigl(\PP H^0\bigl(\PP^5,\OO_{\PP^5}(3)\bigr)\bigr)^G$.
   
It remains to check that a generic element 
  \[  (x,y)\in \bigcup_{0\le r_0,r_1,r_2\le 1}\Gamma_{(\iota_0)^{r_0}\circ (\iota_1)^{r_1}\circ (\iota_2)^{r_2}}\setminus  \Gamma_{\iota_\PP} \]
  also imposes $2$ independent conditions. Let us first assume $(x,y)$ is generic on $\Gamma_{\iota_0}$. Let us write $x=[a_0:a_1:\ldots:a_5]$. By genericity, we may assume all $a_i$ are $\not=0$ (intersections of $\Gamma_{\iota_0}$ with a coordinate hyperplane have codimension $>n+1$ and so need not be considered for hypothesis (\rom1) of proposition \ref{voisin2}). We can thus write
   \[ x=[a_0:a_1:a_2:a_3:a_4:a_5]\ ,\ \ \ y= [-a_0:a_1:a_2: a_3:a_4:a_5] \ ,\ \ \ a_i\not=0 \ .\] 
   The cubic 
   \[  (a_1)^3 (X_0)^3 - (a_0)^3(X_1)^3=0 \]
   is $G$--invariant and contains $x$ while avoiding $y$, and so the element $(x,y)$ again imposes $2$ independent conditions. 
   
  The argument for the other $r_i$ is similar: consider for instance a generic element
   $(x,y)$ in $\Gamma_{\iota_0\circ\iota_1}$. By genericity, we can write
   \[ x=[a_0:a_1:a_2:a_3:a_4:a_5]\ ,\ \ \ y= [-a_0:-a_1:a_2: a_3:a_4:a_5] \ ,\ \ \ a_i\not=0 \ .\]   
   The cubic 
       \[  (a_2)^3 (X_0)^3 - (a_0)^3(X_2)^3=0 \]
   is $G$--invariant and contains $x$ while avoiding $y$, and so the element $(x,y)$ again imposes $2$ independent conditions.

   This proves hypothesis (\rom1) of proposition \ref{voisin2} is satisfied.

  To establish hypothesis (\rom2) of proposition \ref{voisin2}, let $Y=Y_b$ be a cubic as in theorem \ref{main2}, and let us suppose there is a relation
  \[  c\,\Delta_{Y} +d\, \Gamma_{\iota_{}} +\delta =0\ \ \ \hbox{in}\ H^8(Y\times Y)\ ,\]
  where $c,d\in \QQ$ and $\delta\in\ima\bigl( A^4(\PP^5\times\PP^5)\to A^4(Y_b\times Y_b)\bigr)$.
  Looking at the action on $H^{3,1}(Y)$, we find that necessarily $c=d$ (indeed, $\delta$ does not act on $H^{3,1}(Y)$, and $\iota_{}$ acts as minus the identity on $H^{3,1}(Y)$).  
  
     On the other hand, looking at the action on $(H^{4}(Y)_{prim})^\iota$ (which is non--zero thanks to lemma \ref{trace}), we find that
    $c=-d$. We conclude that $c=d=0$, and
       so hypothesis (\rom2) is satisfied.  
  
  \begin{lemma}\label{delorme} Let $P^\prime=\PP(2^3,1^3)$. Let $r,s\in P^\prime$ and $r\not=s$. Then there exists $\sigma\in\PP H^0\bigl(P^\prime,\OO_{P^\prime}(3)\bigr)$ containing $r$ but avoiding 
  $s$.
  \end{lemma}
  
  \begin{proof} It follows from Delorme's work \cite[Proposition 2.3(\rom3)]{Del} that the locally free sheaf $\OO_{P^\prime}(2)$ is very ample. This means there exists $\sigma^\prime\in\PP H^0\bigl(P,\OO_P(2)\bigr)$ containing $r$ but avoiding 
  $s$. Taking the union of $\sigma^\prime$ with a hyperplane avoiding $s$, one obtains $\sigma$ as required.
    \end{proof}

 \begin{lemma}\label{trace} Let $Y\subset\PP^5(\C)$ be a smooth cubic as in theorem \ref{main2}. Then
   \[  \dim \, H^4(Y)^{\iota_Y}\  >1\ .\]
  \end{lemma}
 
 \begin{proof} Griffiths' description of the cohomology of a hypersurface \cite[\S 18]{TH} implies there is an isomorphism, given by the residue map
   \[  H^5(\PP^5\setminus Y)\ \supset\ H_{\le 3}\ \xrightarrow{\cong}\ H^{2,2}(Y)_{prim}\ ,\]
   where $H_{\le 3}$ is by definition the subspace of meromorphic forms with poles of order $\le 3$ along $Y$. To prove the lemma, it thus suffices to exhibit a $\iota_\PP$--invariant meromorphic $5$--form with a pole of order $3$. Let $f=f(X_0,\ldots,X_5)$ be an equation defining $Y$.
   The meromorphic form
   \[      {(X_0)^3\over f^3 } \displaystyle\sum_{j=0}^5   X_j\,   {{dX_0\wedge \cdots \wedge \hat{dX_j}\wedge\cdots\wedge d X_5}}\ \ \in H_{\le 3} \]
   does the job.
   
   Another proof of lemma \ref{trace} (suggested by the referee, whom I thank) is as follows: the cubic $Y$ contains the plane $P:=\{x_3=x_4=x_5=0\}$. The plane $P$  is not proportional to the class $h^2$ where $h\in A^1(Y)$ is a hyperplane class (indeed: if $P$ were equal to $mh^2$ in $H^4(Y)$ for some integer $m$, the intersection $P\cdot h^2$ would be a multiple of $3$, whereas $P\cdot h^2=1$). Since both $P$ and $h^2$ are $\iota_Y$--invariant, this proves the lemma.
   \end{proof}
 \end{proof}

 This closes the proof of lemma \ref{ok}, and hence of theorem \ref{main2}.     
  \end{proof}

   \subsection{Second part}  
     
   \begin{theorem}\label{main}  Let $Y\subset\PP^5(\C)$ be a smooth cubic fourfold defined by an equation
    \[ \begin{split} (X_0)^2 \ell_0(X_3,X_4,X_5)+  (X_1)^2\ell_1(X_3,X_4,X_5)+(X_2)^2\ell_2(X_3,X_4,X_5)+X_0 X_1  \ell_3(X_3,X_4,X_5)& \\ + X_0 X_2 \ell_4(X_3,X_4,X_5) +  
    X_1 X_2 \ell_5(X_3,X_4,X_5)  +g(X_3,\ldots,X_5)=0\ ,& \\
    \end{split}\]
    where the $\ell_i$ are linear forms and $g$ is a homogeneous degree $3$ polynomial.
   
    Let $X=F(Y)$ be the Fano variety of lines in $Y$, and let $\iota\in\aut(X)$ be the anti--symplectic involution of lemma \ref{inv}.
 Then
   \[  \begin{split}  \iota^\ast=-\ide\colon\ \ \ &A^i_{(2)}(X)\ \to\ A^i_{(2)}(X)\ \ \ \hbox{for}\ i=2,4\ ;\\
                            \iota^\ast=\ide\colon\ \ \ &A^4_{(j)}(X)\ \to\ A^4_{(j)}(X)\ \ \ \hbox{for}\ j=0,4\ .\\
                        \end{split}\]    
                  \end{theorem}     
   
   \begin{proof} First, we note that
    \[ A^2_{(2)}(X) = I_\ast A^4_{hom}(X)\ ,\]
    where $I\subset X\times X$ is the incidence correspondence \cite[Proof of Proposition 21.10]{SV}. On the other hand, 
    \[ I={}^t P\circ P\ \ \ \hbox{in}\ A^2(X\times X)\ ,\]
    where $X\leftarrow P\to Y$ denotes the universal family of lines on $Y$ \cite[Lemma 17.2]{SV}. Hence,
    \[ A^2_{(2)}(X) = ({}^t P)_\ast  P_\ast A^4_{hom}(X)\ .\]
    But $P_\ast\colon A^4_{hom}(X)\to A^3_{hom}(Y)$ is surjective \cite{Par}, and so
    \[  A^2_{(2)}(X) = ({}^t P)_\ast A^3_{hom}(Y)\ .\]
    
    It is readily checked that the diagram
    \[ \begin{array}[c]{ccc}
        A^3_{hom}(Y) & \xrightarrow{({}^t P)_\ast}& A^2(X)\\
      {\scriptstyle (\iota_Y)^\ast}\  \downarrow\ \ \ \ \  &&\ \ \ \ \downarrow \ {\scriptstyle \iota^\ast} \\     
      A^3_{hom}(Y) & \xrightarrow{({}^t P)_\ast}& A^2(X)\\
      \end{array}\]
      is commutative (this is because the involution extends to an involution on $P$).
      Using this diagram, theorem \ref{main2} implies that $\iota$ acts as minus the identity on $({}^t P)_\ast A^3_{hom}(Y)= A^2_{(2)}(X)$.
      
      Because the intersection product induces a surjection
      \[ A^2_{(2)}(X)\otimes A^2_{(2)}(X) \to\    A^4_{(4)}(X) \]
      (theorem \ref{chowringfano}(\rom2)), it follows that $\iota$ acts as the identity on $A^4_{(4)}(X)$.
      
      Next, we want to exploit the fact that there is an isomorphism
       \[ \cdot \ell\colon\ \ \ A^2_{(2)}(X)\ \xrightarrow{\cong}\ A^4_{(2)}(X) \]
       (theorem \ref{chowringfano}(\rom1)). Since $\iota^\ast(\ell)=\ell$ (proposition \ref{ell} below), this implies that $\iota$ acts as minus the identity on $A^4_{(2)}(X)$.
       
      \begin{proposition}\label{ell} Let $X$ be the variety of lines on a smooth cubic fourfold $Y\subset\PP^5(\C)$, and let $\iota\in\aut(X)$ be an involution induced by an involution $\iota_Y\in\aut(Y)$. Let $\ell\in A^2(X)$ be the class of theorem \ref{chowringfano}(\rom1).
      Then
        \[ \iota^\ast(\ell)=\ell\ \ \ \hbox{in}\ A^2(X)\ .\]
       \end{proposition} 
      
      \begin{proof} We give two proofs of this fact. The first proof has the benefit of brevity; the second proof will be useful in proving another result (lemma \ref{compat} below).      
      
 \noindent
 {\it First proof:\/} It is known that
   \[ \ell= {5\over 6} c_2(X)\ \ \ \hbox{in}\ A^2(X) \]
   (where the right--hand side denotes the second Chern class of the tangent bundle $T_X$ of $X$) \cite[Equation (108)]{SV}. Since
   \[ \iota^\ast c_2(X) = c_2(\iota^\ast T_X)=c_2(X)\ \ \ \hbox{in}\ A^2(X)\ ,\]
   this proves the proposition.
 
 \noindent
 {\it Second proof:\/} Shen--Vial define the class $L\in A^2(X\times X)$ (lifting the Beauville-Bogomolov class $\BB\in H^4(X\times X)$) as
     \[ L:=   {1\over 3}\bigl(  (g_1)^2  +{3\over 2} g_1 g_2 +(g_2)^2 -c_1 -c_2\bigr) - I         \ \ \ \in A^2(X\times X) \ \]
   \cite[Equation (107)]{SV}. Here $I$ is the incidence correspondence, and
     \[ \begin{split} g&:=-c_1(\EE_2)\ \ \ \in\ A^1(X)\ ,\\
                        c&:=c_2(\EE_2)\ \ \ \in\ A^2(X)\ ,\\
                        g_i&:= (p_i)^\ast(g)  \ \ \ \in\ A^1(X\times X)\ \ \ (i=1,2)\ ,\\
                        c_i&:= (p_i)^\ast(c)  \ \ \ \in\ A^2(X\times X)\ \ \ (i=1,2)\ ,\\ 
                       \end{split}\] 
    where $\EE_2$ is the rank $2$ vector bundle coming from the tautological bundle on the Grassmannian, and $p_i\colon X\times X\to X$ denote the two projections. 
   
   Clearly one has 
   \[ (\iota\times\iota)^\ast (I) =I\ ,\ \ \ (\iota\times\iota)^\ast(c_i)=c_i\ ,\ \ \  (\iota\times\iota)^\ast(g_i)=g_i  \ .\]
   In view of the definition of $L$, it follows that
   \[  (\iota\times\iota)^\ast (L) =L\ \ \ \hbox{in}\ A^2(X\times X)\ .\]
    Using Lieberman's lemma \cite[Lemma 3.3]{V3}, plus the fact that ${}^t \Gamma_\iota= \Gamma_\iota$, this means there is a commutativity relation
     \begin{equation}\label{commut}  L\circ \Gamma_\iota= \Gamma_\iota\circ L\ \ \ \hbox{in}\ A^2(X\times X)\ .\end{equation}
     
     The class $\ell$ is defined as $\ell:=(i_\Delta)^\ast(L)\in A^2(X)$. We now find that
     \[ \begin{split}  \iota^\ast (\ell)&=\iota^\ast (i_\Delta)^\ast(L)    \\
                                    & = (i_\Delta)^\ast (\iota\times \iota)^\ast (L) \\
                                    & =  (i_\Delta)^\ast ( \Gamma_\iota\circ L\circ \Gamma_\iota) \\
                                    & =  (i_\Delta)^\ast (L) =\ell\ \ \ \ \hbox{in}\ A^2(X)\ .\\
                                 \end{split} \]
                  Here the second equality is by virtue of the commutative diagram
            \[  \begin{array}[c]{ccc}
            X & \xrightarrow{i_\Delta} & X\times X\\
            {\scriptstyle \iota} \ \downarrow\ \ \ \ & & \ \ \ \ \downarrow\ {\scriptstyle \iota\times\iota}\\
               X & \xrightarrow{i_\Delta} & X\times X\\
             \end{array}\]  
          The third equality is again Lieberman's lemma, plus the fact that ${}^t \Gamma_\iota=\Gamma_\iota$. The last equality is (\ref{commut}).  
     \end{proof}
     
     It only remains to prove theorem \ref{main} is true for $(i,j)=(4,0)$. This follows from the fact that $A^4_{(0)}(X)$ is generated by $\ell^2$ \cite{SV}, plus the fact that $\ell$ is $\iota$--invariant (proposition \ref{ell}). Theorem \ref{main} is now proven.
   \end{proof}
   
 For later use, we remark that the argument of proposition \ref{ell} also proves the following compatibility statement:  
    
  \begin{lemma}\label{compat} Let $X$ be the variety of lines on a smooth cubic fourfold $Y\subset\PP^5(\C)$, and let $\iota\in\aut(X)$ be an involution induced by an involution $\iota_Y\in\aut(Y)$. Then
    \[ \iota^\ast A^i_{(j)}(X)\ \subset\ A^i_{(j)}(X)\ \ \ \forall i,j\ .\]
     \end{lemma}  
     
   \begin{proof} Let $L\in A^2(X\times X)$ be the Shen--Vial class as above. We observe that equality (\ref{commut}) also implies
    \[ (\iota\times\iota)^\ast (L^r) = \bigl((\iota\times\iota)^\ast (L)\bigr)^r = L^r\ \ \ \hbox{in}\ A^4(X\times X)\ \ \ \forall r\in\NN\ .\] 
   Using Lieberman's lemma, this is equivalent to the commutativity relation
    \[  \Gamma_\iota\circ L^r = L^r\circ \Gamma_\iota\ \ \ \hbox{in}\ A^{2r}(X\times X)\ .\]
  
  Since the Shen--Vial Fourier transform $\FF\colon A^\ast(X)\to A^\ast(X)$ is defined by a polynomial in $L$ \cite[Section C.1]{SV}, we find that
    \[ \FF(\iota^\ast(a))=\iota^\ast \FF(a)\ \ \ \forall a\in A^i(X)\ .\]
    This proves the lemma, for the decomposition of \cite{SV} is defined as
    \[  A^i_{(j)}(X):= \bigl\{ a\in A^i(X)\ \vert\ \FF(a)\in A^{4-i+j}(X)\bigr\}\ .\]
      \end{proof}

 \section{Corollaries}
 
 In this last section, we consider the quotient 
   $ Z:=X/\iota$,
   for $(X,\iota)$ as in theorem \ref{main}. The variety $Z$ is a slightly singular Calabi--Yau variety. As is well--known, Chow groups with $\QQ$--coefficients of quotient varieties such as $Z$ still have a ring structure \cite[Examples 8.3.12 and 17.4.10]{F}. For this reason, we will write $A^i(Z)$ for the Chow group of codimension $i$ cycles on $Z$ (just as in the smooth case).
 
 \begin{corollary}\label{cor}  Let $(X,\iota)$ be as in theorem \ref{main}. Let $Z:=X/\iota$ be the quotient. Then 
 the image of the intersection product map
   \[ A^2(Z)\otimes A^2(Z)\ \to\ A^4(Z) \]
   has dimension $1$.
  \end{corollary}
 
 \begin{proof} 
 We start by establishing a lemma:
 
 \begin{lemma}\label{split} Let $(X,\iota)$ be as in theorem \ref{main}. 
 Then
     \[ A^2(X)^\iota\ \subset\ A^2_{(0)}(X)\ .\]
   \end{lemma}
   
 \begin{proof} Let $c\in A^2(X)^\iota$, and let us write 
   \[ c=c_0 + c_2\ \ \ \hbox{in}\ A^2_{(0)}(X)\oplus A^2_{(2)}(X)\ ,\]
   where $c_j\in A^2_{(j)}(X)$. Since $c$ is $\iota$--invariant, we also have
   \[ c=\iota^\ast(c)= \iota^\ast(c_0) + \iota^\ast(c_2) =  \iota^\ast(c_0) - c_2\ \ \ A^2(X)\ \]
   (here we have used theorem \ref{main} to conclude that $\iota^\ast(c_2)=-c_2$).
   But we know that $\iota^\ast(c_0)\in A^2_{(0)}(X)$ (lemma \ref{compat}), and so (by unicity of the decomposition $c=c_0+c_2$) we must have
    \[ \iota^\ast(c_0)=c_0\ ,\ \ \ -c_2=c_2\ .\] 
   \end{proof}

 Now, let $p\colon X\to Z$ denote the quotient morphism. Thanks to lemma \ref{split}, one has
   \[ p^\ast A^2(Z)\ \subset\ A^2_{(0)}(X)\ .\]
   It follows that
   \[  \begin{split}  p^\ast   \ima \bigl( A^2(Z)\otimes A^2(Z)\to A^4(Z)\bigr) \ &\subset\ \ima \bigl(  p^\ast A^2(Z)\otimes p^\ast A^2(Z)\to A^4(X)\bigr)
     \\
     \ &\subset\ 
         \ima \bigl(   A^2_{(0)}(X)\otimes  A^2_{(0)}(X)\to A^4(X)\bigr)\\
        &\subset\ A^4_{(0)}(X)\oplus A^4_{(2)}(X) 
         \ .\\
         \end{split}\]
         (Here for the last inclusion we have used \cite[Proposition 22.8]{SV}.)
         
        On the other hand, one has
        \[  p^\ast   \ima \bigl( A^2(Z)\otimes A^2(Z)\to A^4(Z)\bigr)\ \subset\ A^4(X)^\iota\ ,\]
        and so (by combining with the above inclusion) we find that
         \[  p^\ast   \ima \bigl( A^2(Z)\otimes A^2(Z)\to A^4(Z)\bigr)\ \subset\  \bigl(A^4_{(0)}(X)\oplus A^4_{(2)}(X) \bigr) \cap 
              A^4(X)^\iota\ .\]     
       
       But we have seen (lemma \ref{compat}) that $\iota$ respects the Fourier decomposition, and so
       \[ \bigl(A^4_{(0)}(X)\oplus A^4_{(2)}(X) \bigr) \cap 
              A^4(X)^\iota =  A^4_{(0)}(X)^\iota \oplus A^4_{(2)}(X)^\iota    \ .\]           
       But $A^4_{(2)}(X)^\iota=0$ (theorem \ref{main}), and so
       \[ A^4(X)^\iota =  A^4_{(0)}(X)^\iota\ .\]
       Since $\ell^2$ generates $A^4_{(0)}(X)$ and is $\iota$--invariant (proposition \ref{ell}), we conclude that
       \[      A^4(X)^\iota =  A^4_{(0)}(X)^\iota =A^4_{(0)})X)\cong \QQ \ .\]     
         \end{proof}

% \begin{corollary}\label{cor1}  Let $(X,\iota)$ be as in theorem \ref{main}. Let $Z:=X/\iota$ be the quotient. Then 
% the map induced by intersection product
%   \[ A^3_{hom}(Z)\otimes A^1(Z)\ \to\ A^4(Z) \] 
%  is the zero map.
%  \end{corollary}
  
%  \begin{proof}  Let $p\colon X\to Z$ denote, once again, the quotient morphism. Since $p^\ast A^i(Z)\subset A^i(X)^\sigma$ and $p_\ast p^\ast$ is multiplication by $2$, it suffices to show that the map (induced by intersection product)
  %  \[ A^3_{hom}(X)^\iota\otimes A^1(X)^\iota\ \to\ A^4(X)^\iota \]
  %  is the zero map. 
   
 %  It is shown in \cite{SV} that $A^3_{hom}(X)=A^3_{(2)}(X)$ and that
  % \[ A^3_{(2)}(X)\cdot A^1(X)\ \subset\ A^4_{(2)}(X)\]
 %  \cite[Proposition 22.6]{SV}. It follows that
%    \[ \ima \Bigl(  A^3_{hom}(X)^\iota\otimes A^1(X)^\iota\ \to\ A^4(X)^\iota\Bigr)\ \subset\     A^4_{(2)}(X)^\iota\ .\]
%    But we know from theorem \ref{main} that 
 %    \[  A^4_{(2)}(X)^\iota=0\ .\]      
%  \end{proof} 

  \begin{corollary}\label{cor3} Let $(X,\iota)$ be as in theorem \ref{main}. Let $Z:=X/\iota$ be the quotient.
 Then the image of the intersection product map
    \[ \ima \bigl(  A^3(Z)\otimes A^1(Z)\xrightarrow{} A^4(Z) \bigr) \]
   has dimension $1$.
   \end{corollary} 
   
  \begin{proof} As above, let $p\colon X\to Z$ denote the quotient morphism. It will suffice to show that
   \[ \ima \bigl(  A^3(X)^\iota\otimes A^1(X)^\iota\xrightarrow{} A^4(X)^\iota \bigr)\ \]
    is of dimension $1$.
    
  There is a decomposition 
    \[ A^3(X)^\iota= A^3_{(0)}(X)^\iota\oplus A^3_{(2)}(X)^\iota\ \]
    (this follows from lemma \ref{compat}).
    Moreover, it is known that
    \[ A^3_{(2)}(X)\cdot A^1(X)\ \subset\ A^4_{(2)}(X) \ \]
    \cite[Proposition 22.6]{SV}.
    But we know (theorem \ref{main}) that $A^4_{(2)}(X)^\iota=0$, and so
    \[ A^3_{(2)}(X)^\iota\otimes A^1(X)^\iota\xrightarrow{} A^4(X)^\iota    \]
    is the zero map. It follows that
    \[    \ima \bigl(  A^3(X)^\iota\otimes A^1(X)^\iota\xrightarrow{} A^4(X)^\iota \bigr) =\ima \bigl(  A^3_{(0)}(X)^\iota\otimes A^1(X)^\iota\xrightarrow{} A^4(X)^\iota 
       \bigr)\ .\]
     To analyze the right--hand side, we observe that
     \[ A^3_{(0)}(X)=A^1(X)_{prim}\cdot A^2_{(0)}(X)\ \]
     (indeed, the inclusion ``$\supset$'' is proven in \cite[Proposition 22.7]{SV}; the inclusion ``$\subset$'' follows from \cite[Remark 4.7]{SV}). This implies that
     \[ \begin{split} \ima \bigl(  A^3_{(0)}(X)^\iota\otimes A^1(X)^\iota\xrightarrow{} A^4(X)^\iota 
       \bigr)  \ &\subset\ \Bigl( A^1(X)_{prim}\cdot A^2_{(0)}(X)\cdot A^1(X)\Bigr) \cap A^4(X)^\iota\ \\
                    &\subset \ \Bigl( A^4_{(0)}(X)\oplus A^4_{(2)}(X)\Bigr) \cap A^4(X)^\iota\ \\
                    & = A^4_{(0)}(X)\ .\\
                    \end{split}     \]
                    Here, the second equality is \cite[Equation (118)]{SV}, and the third equality is theorem \ref{main} combined with lemma \ref{compat}.
          \end{proof}

There is a similar statement for $1$--cycles on the quotient:
 
 \begin{corollary}\label{cor2} Let $(X,\iota)$ be as in theorem \ref{main}. Let $Z:=X/\iota$ be the quotient.
 Then the image of the intersection product map
    \[ \ima \bigl(  A^2(Z)\otimes A^1(Z)\xrightarrow{} A^3(Z) \bigr) \]
  injects into $H^6(Z)$ under the cycle class map.
   \end{corollary} 
   
  \begin{proof} As above, let $p\colon X\to Z$ denote the quotient morphism. We have seen (lemma \ref{split}) that
     \[ p^\ast A^2(Z)\ \subset\ A^2_{(0)}(X)\ ,\]
     and so
     \[ p^\ast \ima \bigl(  A^2(Z)\otimes A^1(Z)\xrightarrow{} A^3(Z)\bigr) \ \subset\  \ima \bigl( A^2_{(0)}(X)\otimes A^1(X)\to A^3(X) \bigr)\ .  \]     
      It is known \cite[Proposition A.7]{FLV} that  
     \[  \ima \bigl( A^2_{(0)}(X)\otimes A^1(X)\to A^3(X)\bigr) \ \subset\  A^3_{(0)}(X)  \ ,\]   
     and so it follows that
     \[    p^\ast \ima \bigl(  A^2(Z)\otimes A^1(Z)\xrightarrow{} A^3(Z) \bigr) \ \subset\ A^3_{(0)}(X)\ .\]
       But $A^3_{(0)}(X)$ injects into cohomology under the cycle class map \cite{SV}.
  ( A quick way of proving this injectivity can be as follows: let $\FF$ be the Fourier transform of \cite{SV}. We have that $a\in A^3(X)$ is in $A^3_{(0)}(X)$ if and only if $\FF(a)\in A^1_{(0)}(X)=A^1(X)$ \cite[Theorem 2]{SV}. Suppose $a\in A^3_{(0)}(X)$ is homologically trivial. Then also $\FF(a)\in A^1(X)$ is homologically trivial, hence $\FF(a)=0$ in $A^1(X)$. But then, using \cite[Theorem 2.4]{SV}, we find that
      \[ {25\over 2} a = \FF\circ \FF(a)=0\ \ \ \hbox{in}\ A^3(X)\ .)\]
  %     Let $N^2 H^4(X)\subset H^4(X)$ denote the $\QQ$--vector space of cycle classes. There is an exact sequence
%     \[  0\to A^2_{(0),hom}(X)  \to A^2_{(0)}(X) \to N^2 H^4(X)\to 0\ .\]
%   Let  $b_1,\ldots, b_r$ be (non--canonical) lifts of a basis of the finite--dimensional $\QQ$--vector space  $N^2 H^4(X)$ to $A^2_{(0)}(X)$, and 
%   let $B\subset A^2_{(0)}(X)$ be the finite--dimensional sub--vector space spanned by the $b_i$.
%    
%   It is known that the intersection product map
%    \[  A^2_{(0),hom}(X)\otimes A^1(X)\ \to\  A^3(X) \]
%    is the zero--map \cite[Proposition 22.4]{SV}. It follows that
%    \[  \ima \bigl( A^2_{(0)}(X)\otimes A^1(X)\to A^3(X)\bigr) = \ima \bigl( B\otimes A^1(X)\to A^3(X)\bigr)\ ,\]
%    which is finite--dimensional.
%    
%    Since $p^\ast$ is injective, this proves the corollary.
     \end{proof}

\vskip1cm
\begin{nonumberingt} Thanks to the referee for helpful comments.
Thanks to Kai and Len and Yoyo for daily moral support.
\end{nonumberingt}

\end{document}